\documentclass[reqno]{amsart}
\usepackage{amsfonts}
\usepackage{amsmath}
\usepackage{amssymb}
\usepackage{graphicx}

 \newtheorem{theorem}{Theorem}[section]
\newtheorem{proposition}{Proposition}[section]

\newtheorem{lemma}{Lemma}[section]
 \newtheorem{remark}{Remark}[section]
\usepackage{hyperref}

\def\eps{\varepsilon}

\newcommand{\dis}{\displaystyle}

\def\vs{{\vskip .2cm}}

\begin{document}

\title[Crime Modeling ]{Global existence of solutions for a chemotaxis-type system arising in crime modeling}
\author[R. Man\'asevich]{}
\email{manasevi@dim.uchile.cl}
\author[Q. H. Phan]{}
\email{phanqh@math.univ-paris13.fr}
\author[P. Souplet]{}
\email{souplet@math.univ-paris13.fr}
\thanks{All three authors were partially  supported by ECOS-CONICYT Grant $\#$ C08E04.
Also,  RM was partially supported by Basal-CMM-Conicyt, Milenio grant-P05-004F,  and Anillo grant ACT-87, and
QHP was partially supported by Basal-CMM-Conicyt}

\maketitle

\centerline{\small Raul MANASEVICH}
\centerline{\small Centro de
Modelamiento Matem\'atico and Departamento de Ingenier\'{\i}a
Matem\'atica,} \centerline{\small  Universidad de Chile,}
\centerline{\small
Casilla 170, Correo 3, Santiago, Chile}
\bigskip

\centerline{\small Quoc Hung PHAN}
\centerline{\small Universit\'e Paris 13, Sorbonne Paris Cit\'e,}
\centerline{\small Laboratoire Analyse, G\'eom\'etrie et Applications, CNRS, UMR 7539,}
\centerline{\small 93430 Villetaneuse, France}

\bigskip
\centerline{\small Philippe SOUPLET}
\centerline{\small Universit\'e Paris 13, Sorbonne Paris Cit\'e,}
\centerline{\small Laboratoire Analyse, G\'eom\'etrie et Applications, CNRS, UMR 7539,}
\centerline{\small 93430 Villetaneuse, France}
\bigskip

We consider a nonlinear, strongly coupled, parabolic system arising in the modeling of burglary in residential areas.
The system is of chemotaxis-type and involves a logarithmic sensivity function and specific interaction and relaxation terms.
Under suitable assumptions on the data of the problem, we give a rigorous proof of the existence of a global and bounded, classical solution,
thereby solving a problem left open in previous work on this model.
Our proofs are based on the construction of approximate entropies and on the use of various functional inequalities.
We also provide explicit numerical conditions for global existence when the domain in a square,
including concrete cases involving values of the parameters
which are expected to be physically relevant.

\section{Introduction}

\subsection{Model and main results}
In a  series of recent papers, models based on partial differential equations have been derived  to study crime, see \cite{BN}, 
\cite{CCM, Pit10, RB, SBB, SDPTBBC}.
Some of these papers are most related to modeling burglary of houses.
A basic issue here is to obtain patterns that describe  
the location of hotspots, and study phenomena such as  appearance and disappearance of  them, 
their stability and movement. 

\smallskip
A very successful model  was obtained recently by Short et al.  in \cite{SDPTBBC}. They first derived
an  agent-based statistical model to study the dynamics of hotspots, taking two sociological effects into account: 
the `broken window effect' and the `repeat near-repeat effect'. 
The first effect refers to the observation that  crime in an area leads to more crime,
and the second one to the observation  that houses  burglarized at some moment have an increased
probability of being burglarized again for some period of time after that moment.

The agent-based model considered relies
on the assumption that criminal agents are walking randomly on a two-dimensional lattice and are committing 
burglaries when encountering an attractive opportunity. 
An attractiveness value is assigned to every house, which measures 
how easily the house can be burgled
 without consequences for the criminal agent.
In addition to be walking randomly, the criminal agents move toward
areas of high attractiveness values. In turn, when a burglary occurs, 
it increases the attractiveness of the house that was burglarized and of those nearby. 
If no additional burglaries occur, then the local attractiveness 
decays toward a constant value.

In a second step, by taking  a suitable limit of the equations for the discrete model,
the authors in \cite{SDPTBBC} obtained the following system of parabolic differential equations
(in adimensionalized form): 
\begin{equation}
\label{original0}
\left\{ 
\begin{array}{ll}
\dis\frac{\partial A}{\partial t}=\eta\triangle A +N A+A^0-A, \\ \\
\dis\frac{\partial  N}{\partial t}=\nabla \cdot\Bigr[\nabla  N-N\nabla \vartheta(A)\Bigl]- N A
+\overline A-A^0.
\end{array}
\right.
\end{equation}
Here and thoughout this article, we denote
\begin{equation}
\label{deftheta}
\vartheta(A)=\chi\log A,\quad\hbox{with $\chi>0$ constant.}
\end{equation}
The model described in \cite{SDPTBBC} involves the specific value $\chi=2$
(we note that the parameter $\chi$ cannot be scaled out by a linear change of
dependent or independent variables).
The functions $A(x; t)$ and $ N(x,t)$ respectively represent
the attractiveness value and the criminal density at position $x$
and time $t$.
The first equation describes the evolution of the attractiveness 
of individual houses to burglary, and the second describes the burglar movement.
Here $\Omega$ is a bounded domain in ${\mathbb R}^2$,
$\eta>0$ is the diffusion rate of attractiveness,
$A^0$ is the intrinsic  attractiveness,
$\overline A$ is a contant that in the equilibrium case represents the average attractiveness.

A first effort to study the  dynamics of the model represented by problem (\ref{original0}) was done in \cite{RB}, where 
a corresponding  initial-boundary value problem with no flux boundary conditions is considered. 
To deal with this problem, the authors in \cite{RB} assume, for
simplicity, that $\Omega$ is a square and under some symmetry conditions are able to  map this problem into
one with periodic problem boundary conditions. Their main result is local existence of a solution,
but global existence is left open.
In addition, and as a simplification of the  model in \cite{SDPTBBC}, they considered a generalized version 
of a Keller-Segel chemotaxis model with the goal of
understanding possible conditions for global existence vs blow-up of solutions in finite time for the original model. 

In \cite{SDPTBBC}, an important hypothesis is that burglars are generated in the model at a constant rate
and leave the lattice immediately after they  have committed a burglary. In \cite{Pit10}, a modified agent-based 
 model for burglary was obtained,
a new condition is introduced to modify these effects. This new condition,  that is refer as 
burglar fatigue, models the effect that if  burglars are sufficiently deterred, they will 
 eventually get tired and will stop looking for houses
to burglar. This consideration introduces changes to the original equations of Short et al. 
Some additional changes to the original equations 
come in by considering in the model that, if too many burglaries occur at some location, 
then burglars are likely to assume that most of the high-value 
is gone or that the owners of the place have implemented more strict security measures.  

As in the original model, by taking the continuous limit to this new agent-based model, 
a new system  of parabolic equations is obtained in \cite{Pit10}. This system, complemented 
with initial data and no-flux boundary conditions,
leads us  to  study the following  initial-boundary value problem
(again in adimensionalized form):
\vskip .05cm
\begin{equation}
\label{original}
\left\{ 
\begin{array}{ll}
 \dis\frac{\partial A}{\partial t}=\eta\Delta A  +\psi N A(1-A)+\tilde{A} -A,
 & \quad x\in\Omega,\ t>0, \\ \\
\dis\frac{\partial N}{\partial t}= \nabla \cdot\Bigr[\nabla N-N\nabla\vartheta(A)\Bigl]+\omega-\omega N,
 & \quad x\in\Omega,\ t>0, \\ \\
 \dis\frac{\partial A}{\partial \nu}=\frac{\partial N}{\partial \nu}=0,
 & \quad x\in\partial\Omega,\ t>0, \\ \\
A(x,0)=A_0(x), \quad N(x,0)=N_0(x),
 & \quad x\in\Omega.
\end{array}
\right.
\end{equation}
Here  again the function $\vartheta$ is defined by (\ref{deftheta}) with $\chi=2$  in  \cite{Pit10}. Moreover,
it is assumed -- as we will do throughout this paper -- that
\begin{equation}\label{GenHyp1}
\hbox{ $\Omega$ is a sufficiently smooth bounded domain of ${\mathbb R}^2$, or a square,}
\end{equation}
\begin{equation}\label{GenHyp2}
\hbox{ $\eta,\psi,\omega,\tilde{A}>0$ are constants}
\end{equation}
and that the initial data satisfy 
\begin{equation}\label{GenHyp3}
\hbox{ $(A_0,N_0)\in H^{1+\beta}(\Omega)\times L^2(\Omega)$
for some $\beta>0$, with $A_0>0$ in~$\overline\Omega$ and $N_0\ge 0$ a.e. in $\Omega$.}
\end{equation}
The outward normal vector on $\partial\Omega$ is denoted by $\nu$.
(We point out that all the functional analytic properties needed for the proofs
remain true in a square, although the latter is not smooth at the corner points.)
Problem (\ref{original}) is locally well posed, the positivity of $A$ being of course understood as part
of the definition of solution (see~Section~2 for details).
\vs

In \cite{Pit10}, some results about linearized stability/instability of the homogeneous steady-states 
and some numerical simulations suggesting the existence of hotspots were given for system~(\ref{original}).
However, the (local and) global existence of solutions was left open.

In this paper, our main goal is to give a rigorous proof of the global existence and boundedness of solutions of system (\ref{original})
under suitable assumptions. In what follows, we denote
$$
\eps_0=\eps_0(\Omega)=\mu^{-2}K^{-1/2},
$$ 
where $\mu$ is the best constant in the Poincar\'e-Sobolev inequality (\ref{PoincSob})
and $K$ is the best constant in the interpolation estimate (\ref{EllEstim}).
Also, recalling that $A_0\in C(\overline\Omega)$ by Sobolev imbedding, we set
$$A_{\max}=\max\{1, \tilde{A},\displaystyle \max_{\overline\Omega} A_0\}
\ge A_{\min}=\min\{1, \tilde{A}, \displaystyle\min_{\overline\Omega} A_0\}>0.$$ 

Our main result is the following:

\begin{theorem}\label{global}
Assume (\ref{deftheta}), (\ref{GenHyp1})-(\ref{GenHyp3}) and
\begin{equation}
\label{small}
 \Bigl(\frac{A_{\max}}{A_{\min}}\Bigr)^2(A_{\max}-A_{\min})\max\{\|N_0\|_1, |\Omega|\}< 
 \eps_0{\hskip 1pt}\eta {\hskip 1pt} \psi^{-1}\chi^{-2}.
\end{equation}
Then the solution of problem (\ref{original}) is global and satisfies the uniform bounds
\begin{equation}
\label{boundAinfty}
\sup_{t\geq 0} \|A(t)\|_{\infty}<\infty
\end{equation}
and
\begin{equation}
\label{boundNinfty}
\sup_{t\geq \tau} \|N(t)\|_{\infty}<\infty,\quad\hbox{ for all } \tau>0.
\end{equation}
\end{theorem}

On the other hand, when $\chi\le 1$, by using an argument from~\cite{Bi2}, we can show that the 
conclusion of Theorem~\ref{global} remains
true under a simple assumption on $A_0$ and $\tilde A$ and without any 
size restriction on $N_0$.
However, this does not apply to the criminological model in~\cite{Pit10} where $\chi=2$.

\begin{theorem}\label{global3}
Assume (\ref{deftheta}), (\ref{GenHyp1})-(\ref{GenHyp3}), with
$0<\chi\le 1$, $ \tilde{A}\le 1$ and $\displaystyle \max_{\overline\Omega} A_0\le 1$. 
Then the solution of problem~(\ref{original}) is global and satisfies the uniform bounds
(\ref{boundAinfty})-(\ref{boundNinfty}).
\end{theorem}

\smallskip

\subsection{Explicit global existence conditions and discussion}

Interestingly, in view of practical applications, the constant appearing in 
Theorem~\ref{global} can be estimated explicitly in the case of a square.

\begin{theorem}\label{global2}
Let $\Omega=(0,L)^2$, with $L>0$. Then the result of Theorem~\ref{global} is true with 
$$\eps_0=\frac{1}{3\sqrt{3}}\approx 0.19.$$
\end{theorem}


With Theorem~\ref{global2} at hand, we shall now explicitly compute our sufficient condition
for global existence and boundedness in concrete cases involving values of the parameters
which are expected to be physically relevant (cf.~\cite{SDPTBBC, Pit10}).
The parameters of system~(\ref{original}) are expressed in~\cite{Pit10} in terms of parameters of the dimensional form of the model.
To avoid confusion, dimensional parameters are marked with a hat in what follows.

The domain is taken to be a square lattice where exactly one house is located at each lattice site.
Distances are thus measured in units of house separations and the adimensional side length of the 
domain is thus essentially equal to $\sqrt{p}$, where $p\gg 1$ is the total number of houses.
After a suitable renormalization (involving the diffusion and mean lifetime parameters of the burglars
-- see table~1 and formula~(3.3) in \cite[p.406]{Pit10}),
the measure of the adimensionalized spatial domain $\Omega$ is given by $|\Omega|=p/3500$. 
We chose $p=3500$, hence $|\Omega|=1$.
Next, the parameter $\omega$ is given by $\omega=\hat{\omega}_2/\hat{\omega}_1$, where 
$\hat{\omega}_1$ and $\hat{\omega}_2$ are the mean lifetimes of the attractiveness and of the active burglars, respectively.
The values chosen in \cite{Pit10} are $\hat{\omega}_1=(1/14)\, {\rm day}^{-1}$ and $\hat{\omega}_2=6\,{\rm day}^{-1}$, hence $\omega=84$.
The parameter $\psi$ is given by $\psi=\hat{\theta}\hat{\Gamma}/\hat{\omega}_1\hat{\omega}_2$, where 
$\hat{\Gamma}$, describing the source term, stands for the number of burglars becoming active 
per time and surface unit and $\hat{\theta}$ is the factor that boosts attractiveness due to a burglary.
We take $\hat{\Gamma}=0.002\ {\rm burglars}\cdot{\rm day}^{-1}\cdot{\rm house\ separation}^{-2}$, 
which is the value taken in~\cite{SDPTBBC} for the analogous parameter (the value $\hat{\Gamma}=0.5$ used in~\cite{Pit10} seems too high,
apparently due to a lack of spatial normalization).
For the domain that we consider, this means a total source term of $7\ {\rm burglars}\cdot{\rm day}^{-1}$.
Concerning $\hat{\theta}$, following~\cite{Pit10}, we take $\hat{\theta}=1\ {\rm burglars}^{-1}\cdot{\rm day}^{-1}$, which leads to 
$\psi=(14/3)\times 10^{-3}\approx 0.0047$.
As for the attractiveness diffusion parameter $\eta$, which models neighboring effects, smaller values mean that 
repeat victimization is more and more likely than near-repeat victimization.
The values taken in~\cite{SDPTBBC} range from $0.01$ to $0.2$, whereas~\cite{Pit10} takes $\eta=0.001$.

Last, let us consider the initial data. 
Concerning the (adimensional) attractiveness $A$, recall 
that the constant $\tilde A$ is assumed to represent its static component,
while $A-\tilde A$ represents the dynamic component associated with the boost hypothesis.
However it is not made completely clear in~\cite{SDPTBBC, Pit10} how $A$ should be evaluated or measured.
Here we shall assume that
\begin{equation}\label{condAtilde}
\tilde A\leq A_0(x)\leq 1,
\end{equation}
which implies that $A_{\min}=\tilde A$ and $A_{\max}=1$.
Consequently (cf.~Lemma~\ref{priori} below), the attractiveness remains for all times larger or equal 
to $\tilde A$, and at most one.
Note that in the numerical simulations from \cite{Pit10}, the initial data
are assumed to be close to the homogeneous steady-state:
$$(A^*,N^*)=\biggl(\frac{\psi-1+\sqrt{(\psi-1)^2+4\psi \tilde A}}{2\psi}, 1\biggr),$$
which in particular satisfies~(\ref{condAtilde}).
As for the initial burglar density $N_0$, we assume for simplicity that its average value $\overline{N_0}$ satisfies
$$ 
\overline{N_0}:=|\Omega|^{-1}\|N_0\|_1\le 1=N^*.
$$  
In dimensional units (cf.~\cite{Pit10}), in view of the above choice of $\hat{\omega}_2$ and $\hat{\Gamma}$, 
this corresponds to an initial total population less than $\approx 1.2$ burglars
(although this may seem small, recall that we have at the same time a total source term of $7\ {\rm burglars}\cdot{\rm day}^{-1}$).

Now, for the above-chosen numerical values of the parameters, the 
global existence and boundedness condition~(\ref{small}) in Theorems~\ref{global} and \ref{global2} becomes
\begin{equation}
\tilde A^{-2}(1-\tilde A) < \gamma\,\eta,\quad\hbox{ where } \gamma
=\frac{1}{12\sqrt 3}{\hskip 1pt} (|\Omega|\psi)^{-1}
=\frac{125\sqrt 3}{21}  \approx 10.31.
\end{equation}
This is equivalent to 
$$\tilde A>\tilde A_-:=\frac{2}{1+\sqrt{1+4\gamma\,\eta}}.$$
The values of $\tilde A_-$ as a function of $\eta$ are given in the following table:
\vspace{5mm}
$$
\vspace{5mm}
\begin{tabular}{|c|c|} 
  \hline
 $\eta$ & $\tilde A_-$ \\
  \hline
  \ \ 0.01\ \   & \ \ 0.91\ \   \\
  \ \ 0.05\ \  & \ \ 0.73\ \   \\
  \ \ 0.1\ \  & \ \ 0.61\ \   \\
  \ \ 0.2\ \  & \ \ 0.49\ \   \\
 \hline
\end{tabular}
$$

Therefore, our global existence conditions are compatible with magnitudes of attractivity which 
are up to about twice those of their static component in the range of $\eta$ used by \cite{SDPTBBC}.
We note that it is suggested in~\cite{Pit10} that a ratio 
of order $10$ (instead of $2$) might be desirable, even for smaller values of $\eta$
that the ones we consider here. However, this is still beyond
the range in which we can rigorously prove global existence
(and actually, global existence for such systems need not be taken for granted -- see next subsection).


\subsection{Related results and comments}

As noted in \cite{SDPTBBC}, problems (\ref{original0}) and (\ref{original}) belong to the family of chemotaxis-type systems. 
A more general class of such systems takes the form
\begin{align}\label{general}
 \begin{cases}
\dis\frac{\partial A}{\partial t}=\eta\Delta A -A +  f(A,N),\\ \\
\dis\frac{\partial N}{\partial t}=\Delta N -\nabla (N\nabla h(A))-\omega N+g(A,N).
 \end{cases}
\end{align}
The typical feature of chemotaxis-type systems is the presence of the antidiffusion term, here 
$-\nabla (N\nabla  h(A))$,
by which the population, of density $N$, tends to move towards higher concentrations of $A$.
The best-known among this class of problems is the Keller-Segel model \cite{KS}, corresponding to $f=N$, $h=A$ and $g=0$.
There is a very large literature on such systems (see \cite{HP, Ho}, for recent surveys).

System (\ref{general}) with a logarithmic sensivity function $h=\chi\log(A)$ 
was considered in~\cite{Bi, SW} for $f=N$ and $g=\omega=0$.
Without any size restriction on the initial data, existence of a global classical solution
was proved in~\cite{Bi} under the restriction $\chi\le 1$,
whereas existence of a (possibly singular) global weak solution was proved in~ \cite{SW} for any $\chi>0$.
System (\ref{general}) with a logarithmic sensivity function was also considered in \cite{AOTYM} (see also \cite{OTYM}),
in conjunction with a quadratic absorption term, namely $g=-cN^2$, and $f=N$.
In this case, global classical existence was proved for any $\chi>0$ without size restriction on the initial data,
but the proof made crucial use of the quadratic absorption term.

Here, with only a linear absorption $-\omega N$ and the specific $f$ under consideration,
a size condition on the initial data is required to prove existence of global classical solution,
but we do not need absolute restrictions on the parameter $\chi$. Although we do not know
what happens for larger initial data, let us recall that in the case of the Keller-Segel model
in two space dimensions, global existence is true only for small initial mass and that blow-up may
occur for large mass. However proving blow-up for $2$-dimensional parabolic chemotaxis models is a very challenging task. 
The only known blow-up result for the $2d$ Keller-Segel system is that of Herrero and Vel\'azquez~\cite{HVpisa},
and its proof, based on matched asymptotics, is highly technical.
In simplified parabolic-elliptic chemotaxis models, where the term $\partial A/\partial t$ in~(\ref{general}) is replaced by $0$,
blow-up proofs are easier and there are more results available. In particular,
when $h=\chi\log A$, $f=N$, $g=\omega=0$ and the space dimension $n$ is at least $3$,
finite time blow-up is known \cite{NS} to occur for $\chi>2n/(n-2)$. For blow-up results concerning the 
parabolic-elliptic Keller-Segel system, see e.g. the survey~\cite{Bl}.

Our global existence proofs, as is customary in chemotaxis-type problems, rely on a priori estimates obtained by energy 
and entropy arguments. 
However, in the case of Theorem~\ref{global}, some care is needed to properly 
take into account the size condition on the initial data.
In the case of the Keller-Segel system, an exact entropy functional was found and used to prove global existence
for suitably small mass in~\cite{NSY97, GZ, Bi}.
Although system (\ref{original}) is not known to possess an exact entropy functional,
{\it assuming condition (\ref{small}) on the data of the problem,}
we can nevertheless construct an {\it approximate} entropy functional (with the typical 
$N\log N$ growth -- cf. Lemma~\ref{lembound}), 
which is the key to our a priori estimates. 
We also stress that, in spite of uniform lower and upper bounds for $A$
which easily follow from the maximum principle, the proof of the estimates of $N$ is far from being immediate.
As for Theorem~\ref{global3}, under the assumptions $\chi\le 1$, $ \tilde{A}\le 1$ and 
$\displaystyle \max_{\overline\Omega} A_0\le 1$, we can use a modified entropy functional from~\cite{Bi2}.

\begin{remark}
{\rm By straightforward modifications of the method (see Remark~\ref{remrem} for more details), one can show that 
a result similar to Theorem~\ref{global}
remains true for the initial-boundary value problem associated with the more general 
system (\ref{general}) 
{\rm where $f, g, h$ are (sufficiently smooth) functions satisfying  the following conditions}
\begin{itemize}
 \item $({\mathcal H}_1)$ \quad  $g(A,0)\geq 0, \quad \forall A\geq 0.$\\
\item $({\mathcal H}_2)$ There exist two positive constants  $A_{\min}< A_{\max}$ such that
\begin{align*}
 -A_{\min}+f(A_{\min}, N)\geq 0, \quad \forall N\geq 0,\\
-A_{\max}+f(A_{\max}, N)\leq 0, \quad \forall N\geq 0.
\end{align*}
\item $({\mathcal H}_3)$ 
$$|g(A,N)|\leq g_1(A)N^{1-\delta}+g_2(A),$$
\qquad where $\delta\in (0,1)$ and $g_1, g_2$ are bounded in $[A_{ \min}, A_{\max}]$.
\item $({\mathcal H}_4)$ 
$$|f(A,N)|\leq f_1(A) N+ f_2(A),$$
\qquad where  $f_1, f_2$ are bounded in $[A_{ \min}, A_{\max}]$.\\
\item $({\mathcal H}_5)$
$$\sup_{[A_{\min}, A_{\max}]}\bigg|\frac{d^ih(A)}{dA^i}\bigg|<\infty, \quad i=1,2.$$
\end{itemize}
}
\end{remark}


The outline of the rest of the article is as follows. 
Subsection 2.1 is devoted to local existence and uniqueness,
whereas Subsection 2.2 contrains basic estimates, energy identities and functional inequalities
which will needed for proving global existence.
Sections 3--5 are then devoted to the proofs of Theorems~1.1, 1.3 and 1.2, respectively.
Finally, our main conclusions are summarized in Section~6.

\section{Preliminaries}

\subsection{Local existence and uniqueness}

We have the following local existence and uniqueness result.

\begin{proposition}
\label{proplocexistP}
Assume (\ref{deftheta}), (\ref{GenHyp1})-(\ref{GenHyp3}) and
 fix any $\beta\in(0,\frac12)$.
Then problem (\ref{original}) admits a unique, classical,
maximal in time solution $(A, N)$ with $A>0$.
Moreover, if its maximal existence time $T^*$ is finite,
then $\lim_{t\to T^*}\|A(t)\|_{H^{1+\beta}}+\|N(t)\|_{L^2}=\infty$.
\end{proposition}

Note that the solution given by Proposition~\ref{proplocexistP} is of course independent of the choice of $\beta\in(0,\frac12)$
(due to local uniqueness and the fact that $H^{1+\beta}(\Omega)\subset H^{1+\gamma}(\Omega)$ for $\beta>\gamma$).
The procedure for proving Proposition~\ref{proplocexistP}
is rather standard (see, e.g., \cite{OTYM, AOTYM, HW05} for more details on similar problems). 
Although the function $\vartheta$ is singular at $A=0$, this
actually causes no difficulty, because a positive lower bound for $A$ can be deduced from the maximum principle
(first working with a smooth replacement of $\vartheta$).
We sketch the proof for completeness.

\smallskip

We first recall some properties of theory of abstract evolution equation. 
Let $\mathcal H, \mathcal V$ be two separable Hilbert spaces with dense and 
compact embedding $\mathcal V\subset\mathcal H$. Identifying $\mathcal H$ 
with its dual $\mathcal H'$ and denoting the dual space of $\mathcal V$ by $\mathcal V'$. 

Consider the initial value problem
\begin{align}\label{2}
 \begin{cases}
  \dis\frac{dU}{dt}+\mathcal AU=F(U), \quad 0<t<\infty,\\
U(0)=U_0.
 \end{cases}
\end{align}
Here, $\mathcal A$ is bounded linear operator from $\mathcal V$ to $\mathcal V'$ 
which is defined by a symmetric bilinear form $a(.,.)$ on $\mathcal V$ satisfying
\begin{align}\label{4}
 &|a(U,V)|\leq M \|U\|_{\mathcal V} \|V\|_{\mathcal V}, \quad U,V\in \mathcal V,\notag\\
& a(U,U)\geq c \|U\|^2_{\mathcal V}, \quad U\in \mathcal V,
\end{align}
with some positive constants $c, M$. $F: \mathcal V\to  \mathcal V' $ is a continuous mapping such that,
for each $\theta>0$, there exist nondecreasing functions $\phi_\theta, \psi_\theta$ such that
\begin{align}
& \|F(U)\|_{ \mathcal V'}\leq \theta\|U\|_{ \mathcal V}+\phi_\theta(\|U\|_{ \mathcal H}), \\
&\|F(U)-F(V)\|_{ \mathcal V'}\leq \theta \|U-V\|_{ \mathcal V}
+(\|U\|_{ \mathcal V}+\|V\|_{ \mathcal V}+1)\psi_\theta(\|U\|_{ \mathcal H}+\|V\|_{ \mathcal H})\|U-V\|_{ \mathcal H},\label{5}
\end{align}
for $U, V\in  \mathcal V$.
Then by standard argument (see \cite{RY01}), the following holds:

\begin{proposition}
\label{proplocexist}
 Assume (\ref{4})-(\ref{5}) and let $M>0$ and $U_0\in \mathcal H$ with $\|U_0\|_{\mathcal H}\le M$.
 Then there exists $T=T(M)>0$ and a unique local-in-time solution $U$ to (\ref{2}) such that
\begin{align}
 U\in C([0,T];\mathcal H)\;\cap \;H^1([0,T]; \mathcal V')\;\cap\; L^2([0,T]; \mathcal V).
\end{align}
\end{proposition}

\medskip
To deduce Proposition~\ref{proplocexistP} from Proposition~\ref{proplocexist}, 
we fix $\beta $ in $(0,\frac12)$ and we set
\begin{align*}
 \mathcal V:=H^{2+\beta}_N(\Omega)\times H^1(\Omega) , \quad  \mathcal H:=  H^{1+\beta}(\Omega)\times L^2(\Omega).
\end{align*}
Then 
\begin{align*}
 {\mathcal V'}= H^{\beta}(\Omega)\mathcal\times(H^1(\Omega))'.
\end{align*}
The linear operator $\mathcal A$ is defined by
\begin{align*}
 \mathcal A=\begin{pmatrix}
    T_1&0\\
0&T_2
   \end{pmatrix},
\end{align*}
where $T_1=-\Delta +1$ is regarded as an operator from $H^{2+\beta}_N(\Omega)$ to 
$H^{\beta}(\Omega)$ and $T_2=-\Delta +\omega$ is 
the Laplace operator equipped with the Neumann boundary condition in $H^1(\Omega)'$. 
Then the bilinear form $a(.,.)$ is defined by
\begin{align*}
 a(U,V)=\left(T_1^{(1+\beta)/2}A_1,T_1^{(1+\beta)/2}A_2\right)_{L^2} + \int_{\Omega}(\nabla N_1\nabla N_2+\omega N_1N_2)dx,
\end{align*}
 where $U=( A_1, N_1)$ and $V=(A_2, N_2)$.

For any fixed $\alpha\in (0,A_{\min})$, we pick a function $\vartheta_\alpha\in C^\infty({\mathbb R})$ such that
$\vartheta_\alpha(s)=\chi\log s$ for $s\ge A_{\min}$.
Define $F_\alpha(U)$ by
\begin{align*}
 F_\alpha(U)=\begin{pmatrix}
\psi NA(1-A)+\tilde{A}\\
-\nabla(N\nabla\vartheta_\alpha(A))+\omega
      \end{pmatrix}, \qquad U=(A, N). 
\end{align*}
and denote by $(P_\alpha)$ the modified problem (\ref{original}).

As a consequence of Proposition~\ref{proplocexist}, problem $(P_\alpha)$ admits a unique, 
maximal in time solution $(\hat A, \hat N)$, defined in an interval $[0, T^*)$.
Moreover, if $T^*<\infty$ then $\lim_{t\to T^*}\|\hat A(t)\|_{H^{1+\beta}}+\|\hat N(t)\|_{L^2}=\infty$.
Furthermore, it is classical for $t>0$.
This follows from a standard bootstrap argument based on parabolic regularity and imbedding theorems.

Observing that $A_{\min}$ is a subsolution of the equation for $\hat A$,
we deduce that $\hat A\ge A_{\min}$ in $(0,T^*)$, so that $(\hat A, \hat N)$
actually solves the original problem (\ref{original}).

As for local uniqueness for (\ref{original}), the notion of solution to (\ref{original}) implies that $A$ is uniformly positive
on $\overline\Omega\times [0,T]$ for any $0<T<T^*$ (since $A\in C([0,T^*);H^{1+\beta}(\Omega))$.
Since $\alpha$ can be arbitrary close to $0$ in problem $(P_\alpha)$,
the uniqueness for the original problem follows from the uniqueness for $(P_\alpha)$.

\subsection{Basic estimates, energy identities and functional inequalities}

Our first lemma provides the primary $L^1$ control of $N$ and uniform lower and upper bounds for $A$.
The latter can be easily obtained from the maximum principle,
owing to the special form of the nonlinear term in the equation for $A$.

\begin{lemma}[A priori estimates for $A$ and $N$]\label{priori}
For all $t\in (0,T^*)$, we have
\begin{align}
&A(x,t)\geq \min\{1, \tilde{A}, \inf A_0(x)\}:=A_{\min}, \label{ineqAmin} \\
&A(x,t)\leq \max\{1, \tilde{A}, \sup A_0(x)\}:=A_{\max}, \label{ineqAmax} \\
&N(x,t)\geq 1-e^{-\omega t}>0,\label{Npos} \\
&\|N(t)\|_1=e^{-\omega t}\|N_0\|_1+|\Omega|(1-e^{-\omega t}), \label{ineqN1}
\end{align}
hence in particular
\begin{equation}\label{ineqN1b}
\|N(t)\|_1\leq \max\{\|N_0\|_1, |\Omega|\}:=N_{1,\max}.
\end{equation}\end{lemma}

\begin{proof}
To check (\ref{ineqAmin}) and (\ref{ineqAmax}), 
it suffices to observe that $A_{\min},A_{\max}$ are, respectively, sub-/super\-solution of the equation for $A$,
with Neumann boundary conditions, and to apply the maximum principle.
The proof of (\ref{Npos}) is similar, in view of the regularity of the solution.
As for (\ref{ineqN1}), it follows by integrating the equation for $N$ in space.
\end{proof}

We next state the basic energy identities.

\begin{lemma}[Energy identities]\label{energy}
For all $t\in (0,T^*)$, we have
\begin{align}
&\frac12\frac{d}{dt}\|A(t)\|_2^2+\|A(t)\|_2^2+ \eta \|\nabla A(t)\|^2_2=\psi\int_{\Omega}NA^2(1-A)dx
+\int_{\Omega}\tilde{A}A\,dx,\label{id1}\\
&\frac12\frac{d}{dt}\|\nabla A(t)\|_2^2+\|\nabla A(t)\|_2^2+ \eta \|\Delta A(t)\|^2_2
=-\psi\int_{\Omega}NA(1-A)\Delta A\,dx,\label{id2}\\
&\frac{d}{dt}\int_{\Omega}(N\log N-N+1)dx+\omega \int_{\Omega}(N\log N-N+1)dx
+\int_{\Omega}\frac{|\nabla N|^2}{N}dx
-\int_{\Omega}\nabla N.\nabla\vartheta(A)dx\notag\\
&\qquad\qquad\qquad\qquad\qquad\qquad\qquad\qquad\qquad 
=\omega\int_{\Omega}(\log N-N+1)dx\le 0,\label{id3}\\
&\frac12\frac{d}{dt}\|N(t)\|_2^2+\omega \|N(t)\|_2^2+\|\nabla N(t)\|_2^2
=\int_{\Omega}N\nabla N.\nabla\vartheta(A)dx+\omega\int_{\Omega}Ndx.\label{id4}
\end{align}
\end{lemma}

\begin{proof}
Formulae (\ref{id1})-(\ref{id4}) follow by integration by parts after multipying, 
respectively, the first equation in (\ref{original})  by $A$ and $-\Delta A$,
and the second equation in (\ref{original})  by $\log N -1$ and $N$. Note that these manipulations are licit
in view of (\ref{Npos}), of the classical regularity of $N$ and of the 
higher parabolic regularity applied to the equation for $A$.
\end{proof}

\begin{lemma}[Poincar\'e's inequality]\label{leminequa}
For any uniformly positive function $u\in H^1(\Omega)$, there holds 
\begin{align}\label{PoincSob0}
\|u\|_2^2\leq \mu^2\|u\|_1\int_{\Omega}\frac{|\nabla u|^2}{u}dx+|\Omega|^{-1}\|u\|_1^2,
\end{align}
where $\mu$ is the best constant of the following Poincar\'e-Sobolev inequality:
\begin{equation}
\label{PoincSob}
\|u-|\Omega|^{-1}\int_{\Omega}u\,dx\|_2\leq \mu \|\nabla u\|_1,\quad u\in W^{1,1}(\Omega).
\end{equation}
\end{lemma}

\begin{proof}
Denote $\overline u=|\Omega|^{-1}\int_{\Omega}u\,dx$.
Using the orthogonality of $u-\overline u$ and $\overline u$ in $L^2$ and (\ref{PoincSob}), we have
\begin{align}
\|u\|_2^2&= \|u-\overline u\|_2^2+|\Omega| \overline u^2 \notag \\
&\leq \mu^2 \|\nabla u\|_1^2+|\Omega| \overline u^2. \label{PoincSob2}
\end{align}
Since $\|\nabla u\|_1^2\le \|u\|_1\int_\Omega \frac{|\nabla u|^2}{u}dx$ by the Cauchy-Schwarz inequality, we deduce
(\ref{PoincSob0}).
\end{proof}

A key ingredient in the proof of our main result is the following interpolation estimate.
We recall that $H^2(\Omega)\subset C(\overline\Omega)$, since we are in two space dimensions;
however Lemma~\ref{grad} remains true in any dimension if we assume $u\in H^2(\Omega)\cap C(\overline\Omega)$.

\begin{lemma}\label{grad} 
Assume that $u\in H^2(\Omega)$ satisfies $\partial u/\partial \nu=0$ on $\partial\Omega$
(in the sense of traces).
Then 
\begin{equation}
\label{EllEstim}
\int_{\Omega}|\nabla u|^4dx \leq K\,{\rm osc}^2(u)\int_{\Omega}|\Delta u|^2dx,
\end{equation}
where $K=K(\Omega)>0$.

\end{lemma}
\begin{proof}
{\bf Step 1.} We first give a homogeneous version of the standard elliptic $L^2$-estimate, namely:
\begin{align}\label{elliptic}
 \|D^2u\|_2:=\Bigl(\sum_{ij}\|u_{x_ix_j}\|_2^2\Bigr)^{1/2} \leq C(\Omega)\|\Delta u\|_2,
\end{align}
for any $u\in H^ 2(\Omega)$ with Neumann boundary condition.

To verify (\ref{elliptic}), we start from 
$\|D^2u\|_2\le C(\Omega) \|-\Delta u+u\|_2$,
which is well known (see e.g. \cite[Chapter 9]{Bre83}), hence
\begin{align}\label{1a}
\|D^2u\|_2\le C(\Omega) (\|\Delta u\|_2+\|u\|_2).  
\end{align}
Let $(e_k)_{k\ge 0}$ be a Hilbert basis of $L^2$ made of eigenfunctions of $-\Delta$ with 
domain $H^2$ equipped with Neumann conditions.
Denote the eigenvalues by $\lambda_0=0<\lambda_1\le \lambda_2\le ...$ and observe that $e_0=C=Const.$
Set $f=-\Delta u$, decompose $f=\sum_{k\ge 0} c_k e_k$, and note that 
$$c_0=(f,e_0)=C\int_\Omega f\, dx=-C\int_{\partial\Omega}\partial_\nu u\, d\sigma=0.$$
Let $v:=\sum_{k\ge 1} \lambda_k^{-1}c_k e_k$. Then  $-\Delta v=f$ with $\partial_\nu v=0$ and $v$ satisfies
\begin{align}\label{2a}
\|v\|_2^2=\sum_{k\ge 1} \lambda_k^{-2}|c_k|^2\leq \lambda_1^{-2}\|\Delta u\|_2^2.
\end{align}
Since the difference $z=v-u$ satisfies $\Delta z=0$, $\partial_\nu z=0$, 
it follows that $\int_\Omega |\nabla z|^2\, dx=0$, so that
$$u=v+Const.$$
Combining this with (\ref{1a}) and (\ref{2a}), we  obtain
$$\|D^2u\|_2=\|D^2v\|_2\le  C(\Omega) (\|\Delta v\|_2+\|v\|_2),$$
and
$$\|D^2u\|_2 \leq C(\Omega) (1+\lambda_1^{-1})\|\Delta u\|_2.$$
Hence (\ref{elliptic}).

{\bf Step 2.} By density, it suffices to prove the Lemma for $u\in C^2(\overline\Omega)$.
We may assume without loss of generality that $\min_\Omega u=0$ and  ${\rm osc}(u)=\|u\|_\infty$.
For a matrix $M=(m_{ij})$, we denote $|M|_1=\max_i \bigl\{\sum_j |m_{ij}|\bigr\}$.
Observing that
$$\nabla\cdot(|\nabla u|^2\nabla u)=|\nabla u|^2\Delta u+2\,{}^t\nabla u(D^2 u)\nabla u
={}^t\nabla u\bigl(2 D^2 u+(\Delta u)I\bigr)\nabla u,$$
where $I$ is the identity matrix,
we have
$$\bigl| \nabla\cdot(|\nabla u|^2\nabla u)\bigr|
\leq\bigl| 2 D^2 u+(\Delta u)I\bigr|_1|\nabla u|^2.$$
Using the divergence theorem, it follows that
\begin{align*}
\int_{\Omega}|\nabla u|^4dx&=\int_{\Omega}\nabla u\cdot|\nabla u|^2\nabla u\,dx
=-\int_{\Omega}u\nabla(|\nabla u|^2\nabla u)dx\\
&\leq \int_{\Omega} u\bigl|2D^2 u+(\Delta u)I\bigr|_1|\nabla u|^2 dx
\leq \frac12\int_{\Omega}|\nabla u|^4dx+\frac12\int_{\Omega}u^2\bigl|2D^2 u+(\Delta u)I\bigr|_1^2dx.
\end{align*}
Consequently,
\begin{equation}\label{nabla4}
\int_{\Omega}|\nabla u|^4dx\leq \|u\|_\infty^2\int_{\Omega}\bigl|2D^2 u+(\Delta u)I\bigr|_1^2 dx.
\end{equation}
Since $\int_{\Omega}\bigl|2D^2 u+(\Delta u)I\bigr|_1^2 dx\le C\|D^2 u\|_2^2$, it follows from (\ref{elliptic}) that
\begin{equation}\label{nabla4b}
\int_{\Omega}\bigl|2D^2 u+(\Delta u)I\bigr|_1^2 dx\le K(\Omega)\|\Delta u\|_2^2
\end{equation}
which, along with (\ref{nabla4}), implies the conclusion.
\end{proof}

\medskip

We shall also need the following classical smoothing properties of the Neumann heat semigroup.
For $d,\lambda>0$, we define the operator ${\mathcal A}={\mathcal A}_{d,\lambda}
= -d \Delta +\lambda$ on $L^2(\Omega)$,  with domain 
$D({\mathcal A})=\{v\in H^2(\Omega);\ \partial A/\partial \nu=0 \hbox{ on $\partial\Omega$
(in the sense of traces)}\}$.
We denote by $T(t)=T_{d,\lambda}(t)$ the semigroup generated by ${\mathcal A}$. It is well known that,
for any $0\le m\le 2$, $1\le p\le q\le \infty$ and any $\phi\in L^2(\Omega)\cap L^p(\Omega)$, we have
\begin{equation}
\label{SmoothingEstim}
\|T(t)\phi\|_{W^{m,q}(\Omega)}\leq Ce^{-\lambda t}\bigl(1+t^{-\frac{m}{2}-\frac{1}{p}
+\frac{1}{q}}\bigr)\|\phi\|_{L^p(\Omega)},\qquad t\ge 0,
\end{equation}
with $C=C(\Omega,m,p,q,\eta)>0$.

\section{Global existence and boundedness: Proof of Theorem~\ref{global}}

The following lemma is the key to the proof of our main result. It provides an approximate entropy functional $\phi$,
which is available whenever condition (\ref{small}) is satisfied.
This function enables us to get an $H^1$-bound for $A$ and a time-averaged $H^2$-bound.

\begin{lemma}[Approximate entropy functional]\label{lembound}
Assume (\ref{small}) and let
\begin{equation}
\label{condbarmu}
c_1:=\frac{\eta}{2}\Bigl(1-K\chi^4\eta^{-2}\mu^4\psi^2N_{1,\max}^2A^{-4}_{\min}A^4_{\max}(A_{\max}-A_{\min})^2\Bigr)>0.
\end{equation}

(i) Then the function
$$\phi(t):=\sigma \int_{\Omega}(N\log N-N+1)dx + \frac{1}{2} \|\nabla A(t)\|_2^2,
\qquad\mbox{ with } \sigma=\frac{2\psi^2}{\eta} A^4_{\max}  \mu^2N_{1,\max},$$
satisfies the differential inequality
\begin{equation}
\label{ineqphi}
\phi'+\tilde\omega\phi +c_1\|\Delta A\|_2^2\le c_2,\quad 0<t<T^*,
\end{equation}
where
$\tilde\omega=\min(\omega,2)$ and $c_2=c_2(K,\chi,A_{\min},A_{\max},\psi,N_{1,\max},\eta, |\Omega|, \mu)>0$.

(ii) We have
\begin{equation}
\label{ineqNablaA}
\sup_{t\in (0,T^*)}\|\nabla A(t)\|_2^2\leq c_3,
\end{equation}
with $c_3:=2\max(\phi(0),c_2 \tilde\omega^{-1})$, and
\begin{equation}
\label{ineqDeltaA}
\int_s^t\|\Delta A(\tau)\|_2^2\,d\tau\leq c_4(1+t-s),\quad 0<s<t<T^*,
\end{equation}
with $c_4= c_1^{-1}\max(c_2,c_3/2)$.
\end{lemma}

\begin{proof}
Set
\begin{equation}\label{defphi12}
\phi_1(t):= \|\nabla A(t)\|_2^2,\qquad \phi_2(t):= \int_{\Omega}(N\log N-N+1)dx \ge 0
\end{equation}
(due to $s\log s-s+1\ge 0$, $s>0$).
In this proof, $C$ will denote a generic positive constant depending only on 
$K, \chi, A_{\min}, A_{\max}, \psi$, $N_{1,\max}, \eta,
|\Omega|, \mu$
and on $\eps_1,\eps_2,\eps_3$ below.

On the one hand, it follows from (\ref{id2}) that, for any $\eps_1>0$,
\begin{align}
 \frac12\phi_1'+\phi_1+\eta\|\Delta A\|_2^2
 \leq\eps_1\|\Delta A\|_2^2+\frac{\psi^2}{4\eps_1}\int_{\Omega}N^2A^2(1-A)^2dx. 
\label{ineqPhi1}
\end{align}
Since
\begin{equation}
\|N\|_2^2\le \mu^2N_{1,\max}\int_{\Omega}\frac{|\nabla N|^2}{N}dx+C
\label{ineqN2}
\end{equation}
due to Lemma~\ref{leminequa}, we deduce from (\ref{ineqPhi1}) that
\begin{align}
 \frac12\phi_1'+\phi_1
 \le a_1 \int_{\Omega}\frac{|\nabla N|^2}{N}dx + a_2\|\Delta A\|_2^2  + C
  \label{c1}
\end{align}
where 
\begin{equation}
\label{defa12}
a_1=\frac{\psi^2}{4\eps_1} A^4_{\max}  \mu^2N_{1,\max}>0,
\qquad a_2=\eps_1-\eta.
\end{equation}

On the other hand, it follows from (\ref{id3}) that, for any $\eps_2,\eps_3>0$,
\begin{align*}
\phi_2'+\omega\phi_2+\int_{\Omega}\frac{|\nabla N|^2}{N}dx
&\le \int_{\Omega}\nabla N\nabla\vartheta(A)dx
\leq \eps_2\int_{\Omega}\frac{|\nabla N|^2}{N}dx+\frac1{4\eps_2}\int_{\Omega}N|\nabla \vartheta(A)|^2dx\\
&\leq \eps_2\int_{\Omega}\frac{|\nabla N|^2}{N}dx +\frac1{4\eps_2}\Bigl(\eps_3 \|N\|_2^2
+\frac1{4\eps_3}\int_{\Omega}|\nabla \vartheta(A)|^4dx\Bigr).
\end{align*}
Using
$$\int_{\Omega}|\nabla \vartheta(A)|^4dx\le K \chi^4A_{\min}^{-4}(A_{\max}-A_{\min})^2\|\Delta A\|_2^2$$
due to Lemma~\ref{grad}, and (\ref{ineqN2}), we deduce that
\begin{equation}
\phi_2'+\omega\phi_2
\le a_3 \int_{\Omega}\frac{|\nabla N|^2}{N}dx+a_4\|\Delta A\|_2^2+C,
\label{c2}
\end{equation}
where 
\begin{equation}
\label{defa34}
a_3=\eps_2+\frac{\eps_3\mu^2N_{1,\max}}{4\eps_2}-1, \qquad
a_4=\frac{K \chi^4}{16\eps_2\eps_3}\,A_{\min}^{-4}(A_{\max}-A_{\min})^2>0.
\end{equation}

Now setting $\tilde\omega=\min(\omega,2)$ and combining (\ref{c1}) and (\ref{c2}), we 
see that $\phi=\frac12\phi_1+\sigma\phi_2$ satisfies
\begin{equation}\label{ineqphi12}
\phi'+\tilde\omega\phi 
\le (a_1+\sigma a_3)\int_{\Omega}\frac{|\nabla N|^2}{N}dx\\
+\bigl(a_2+\sigma a_4\bigr)\|\Delta A\|_2^2+C.
\end{equation}
Assume $a_3<0$ and choose $\sigma=-a_1/a_3>0$. Then we have $a_2+\sigma a_4<0$ provided
$a_1a_4<a_2a_3$, that is
\begin{equation}\label{conda1234}
K \chi^4A_{\min}^{-4}A^4_{\max}(A_{\max}-A_{\min})^2
\psi^2   \mu^2N_{1,\max} <
16\eps_1(\eta-\eps_1)\bigl(4(1-\eps_2)\eps_2-\eps_3\mu^2N_{1,\max}\bigr)\eps_3.
\end{equation}
The best condition, maximizing the RHS in (\ref{conda1234}), is obtained by choosing $\eps_1=\eta/2$, $\eps_2=1/2$
and then $\eps_3=(2\mu^2N_{1,\max})^{-1}$, which in turn implies $a_3=-1/4<0$. 
Inequality (\ref{conda1234}) is then equivalent to
$$
K\chi^4A_{\min}^{-4}A^4_{\max}(A_{\max}-A_{\min})^2
\psi^2   \mu^2N_{1,\max}<4
\eta^2(1-\eps_3\mu^2N_{1,\max})\eps_3
=\eta^2(\mu^2N_{1,\max})^{-1},
$$
which is true, due to (\ref{condbarmu}). Then we have $a_2=-\eta/2$,
$$
a_1=\frac{\psi^2}{2\eta} A^4_{\max}  \mu^2N_{1,\max},
\quad a_4=\frac{K\chi^4}{4} \mu^2N_{1,\max} \,A_{\min}^{-4}(A_{\max}-A_{\min})^2,\quad
\sigma=\frac{2\psi^2}{\eta} A^4_{\max}  \mu^2N_{1,\max}.
$$
We conclude from (\ref{ineqphi12}) that (\ref{ineqphi}) holds with 
$c_1=-a_2-\sigma a_4$, which yields the value given in (\ref{condbarmu}).

\smallskip
(ii) Multiplying (\ref{ineqphi}) with $e^{\tilde\omega t}$ and integrating between $0$ and $t$, we obtain 
\begin{equation}\label{boundphi}
\phi(t)\le \max(\phi(0),c_2 \tilde\omega^{-1}),\quad 0<t<T^*.
\end{equation}
This guarantees (\ref{ineqNablaA}), in view of (\ref{defphi12}).
Inequality (\ref{ineqDeltaA}) then follows after integrating (\ref{ineqphi}) over $(s,t)$, taking (\ref{boundphi})
into account.
\end{proof}

Building on estimate (\ref{ineqDeltaA}) from the previous lemma, 
we shall now derive uniform estimates for $\|N(t)\|_2$ and $ \|A(t)\|_{H^m(\Omega)}$,
which in turn will guarantee the global existence of the solution.

\begin{lemma}
\label{bounda}
Assume that (\ref{ineqDeltaA}) is satisfied for some $c_4>0$. Then we have
\begin{equation}
\label{L2bound}
\sup_{t\in (0,T^*)} \|N(t)\|_2<\infty
\end{equation}
and, for each $m\in (0,2)$ and $\tau\in (0,T^*)$,
\begin{equation}
\label{Hmbound}
\sup_{t\in (\tau,T^*)} \|A(t)\|_{H^m(\Omega)}<\infty.
\end{equation}
\end{lemma}

\begin{proof}
Let us first establish (\ref{L2bound}).
It follows from (\ref{id4}) that
\begin{align*}
\frac{d}{dt}\|N(t)\|_2^2+2\omega\|N\|_2^2+2\|\nabla N(t)\|_2^2&\leq 2\,\omega N_{1;\max}
+ 2\int_{\Omega}N\nabla N.\nabla\vartheta(A)dxdt\\
&\leq 2\,\omega N_{1;\max}+\int_{\Omega}\nabla N^2.\nabla\vartheta(A)dx\\
&\leq 2\,\omega N_{1;\max}-\int_{\Omega} N^2.\Delta\vartheta(A)dx\\
&\leq 2\,\omega N_{1;\max}+\|N\|_4^2\|\Delta\vartheta(A)\|_2.
\end{align*}
By the Sobolev embedding of $W^{1,1}(\Omega)$ into $L^2(\Omega)$ with constant $C_S$, we have,
for any $\eps_1>0$,
\begin{align*}
 \|N\|_4^2&=\|N^2\|_2\leq C_S(\|\nabla N^2\|_1+\|N^2\|_1)= C_S(2\|N\nabla N\|_1+\|N\|_2^2)\\
&\leq C_S(2\|N\|_2\|\nabla N\|_2+\|N\|_2^2)\leq \eps_1 \|\nabla N\|_2^2+ (C_S+ C_S^2/\eps_1)\|N\|_2^2.
\end{align*}
Choosing $\eps_1= \|\Delta\vartheta(A)\|_2^{-1}$, we get
\begin{align}\label{c5}
\frac{d}{dt}\|N(t)\|_2^2+2\omega\|N\|_2^2+\|\nabla N(t)\|_2^2\leq 
2\,\omega N_{1;\max} + \big(C_S\|\Delta\vartheta(A)\|_2+ C^2_S\|\Delta\vartheta(A)\|^2_2\big)\|N\|_2^2.
\end{align}
Pick now $\eps>0$, and denote by $C(\eps)$ a generic positive constant depending on the solution and on~$\eps$,
 but independent of $t\in (0,T^*)$.
By the Gagliardo-Nirenberg inequality
$$\|N\|_2\leq C_G(\|\nabla N\|_2^{1/2}\|N\|_1^{1/2}+\|N\|_1),$$
we have
$$\|N\|^2_2\leq 2C_G^2(\|\nabla N\|_2\|N\|_1+\|N\|_1^2)\leq \eps \|\nabla N\|_2^2+ C(\eps)N_{1;\max}^2.$$
Therefore, using (\ref{ineqN1b}), we have
$$ \|\nabla N\|_2^2\geq \frac{1}{\eps}\|N\|^2_2-C(\eps).$$
Combining this with (\ref{c5}), we obtain
\begin{align*}
\frac{d}{dt}\|N(t)\|_2^2+\frac{1}{\eps}\|N\|_2^2&\leq C(\eps) + \big(1+2C^2_S\|\Delta\vartheta(A)\|^2_2\big)\|N\|_2^2.
\end{align*}

Now setting
$$\varphi(t)=\|N(t)\|_2^2,
\qquad \rho(t)=\int_{0}^{t}\Bigl(\frac{1}{\eps}-1-2C^2_S\|\Delta\vartheta(A)\|^2_2\Bigr)ds,$$
we are thus left with the differential inequality
$$\varphi'(t)+ \rho'(t)\varphi(t)\leq C(\eps).$$
On the other hand, we have $\Delta\vartheta(A)=\chi(A^{-1}\Delta A-A^{-2}|\nabla A|^2)$,
hence 
$$\|\Delta\vartheta(A)\|^2_2 \le 2\chi^2\bigl(A^{-2}\|\Delta A\|^2_2+A^{-4}\|\nabla A\|_4^4 \bigr).$$
It then follows from (\ref{ineqDeltaA}),  (\ref{ineqAmin}) and Lemma~\ref{grad} that
$$\int_{s}^{t}\|\Delta\vartheta(A)\|^2_2\leq c_5(1+t-s),\quad 0<s<t<T^*.$$
 By choosing $\eps=(2+2C^2_Sc_5)^{-1}$ and letting $c_6=2C^2_Sc_5$, we have
$$\rho(t)-\rho(s) \ge t-s-c_6,\quad 0<s<t<T^*.$$
By integration, we get
$$\varphi(t)
\leq \varphi(0)e^{-\rho(t)}+ C\int_{0}^{t}e^{\rho(s)-\rho(t)}ds 
\leq \varphi(0)e^{c_6-t}+C\int_{0}^{t}e^{c_6+s-t}ds \leq e^{c_6}(\varphi(0)+C)
$$
and (\ref{L2bound}) follows.

\medskip

To prove (\ref{Hmbound}), we rewrite the first equation via the variation-of-constants formula
\begin{align*}
 A(t)=T_{\eta,1}(t)A_0+\int_{0}^{t}T_{\eta,1}(t-s)(\psi NA(1-A)+\tilde{A})(s) ds
\end{align*}
(where $T_{\eta,1}$ is defined at the end of Section~2).
By (\ref{L2bound}), we have 
$$M:=\sup_{t\in (0,T^*)} \|\bigl(\psi NA(1-A)+\tilde{A}\bigr)(t)\|_2<\infty.$$
Fix $m\in(0,2)$. It follows from (\ref{SmoothingEstim}) that, for any $0<\tau\le t<T^*$,
\begin{align*}
 \|A(t)\|_{H^m(\Omega)}&\leq Ce^{-t}(1+t^{-m/2})\|A_0\|_2+C\int_{0}^{t}e^{-(t-s)}
 (1+(t-s)^{-m/2})\|\bigl(\psi NA(1-A)+\tilde{A}\bigr)(s)\|_2 ds\\
&\leq C (1+\tau^{-m/2})\|A_0\|_2+CM\int_{0}^{\infty}e^{-s}(1+s^{-m/2})ds =: C(\tau),
\end{align*}
which proves (\ref{Hmbound}).
\end{proof}

{\it Proof of Theorem~\ref{global}.}
\smallskip

{\bf Step 1.}  {\it Global existence.}
In view of the local theory stated in Section~2, this is a direct consequence of estimates (\ref{L2bound}) 
and (\ref{Hmbound}) in Lemma~\ref{bounda}.
\medskip

{\bf Step 2.}  {\it Boundedness.}
Estimate (\ref{boundAinfty}) was already obtained in Lemma~\ref{bounda}.
Starting from the global estimates obtained in Lemma~\ref{bounda}, 
we shall use  a standard boostrap argument (see e.g., \cite{NSY97}) to prove (\ref{boundNinfty}).

As a consequence of (\ref{Hmbound}) and Sobolev imbeddings,
we have, for each $\tau>0$,
\begin{equation}
\label{NablaApbound}
\sup_{t\in (\tau,\infty)} \|\nabla A(t)\|_p<\infty,\quad 1\le p<\infty.
\end{equation}
Using (\ref{ineqAmin}) and (\ref{L2bound}), it follows that
\begin{equation}
\label{Mqbound}
M_q:=\sup_{t\in (\tau,\infty)} \|(NA^{-1}\nabla A)(t)\|_q<\infty,\quad 1\le q<2.
\end{equation}
Set $T(t)=T_{1,\omega}(t)$. From the second equation, using $T(t)1=e^{-\omega t}$, we have 
\begin{align*}
N(t)
&=T(t)N_0+\int_{0}^{t}T(t-s)\big(-\nabla(N\nabla\vartheta(A))\big)ds +\int_{0}^{t}T(t-s)\omega ds\\
&=T(t)N_0+1-e^{-\omega t}+\int_{0}^{t} \nabla T(t-s) \cdot NA^{-1}\nabla A ds.
\end{align*}
After a time-shift,
combining this with (\ref{SmoothingEstim}), (\ref{L2bound}) and (\ref{Mqbound}),  
we obtain, for any $q\in [1,2)$, $m\in [0,1)$ and $t\ge \tau/2$,
\begin{align}
 \|N(\textstyle\frac{\tau}{2}+t)\|_{W^{m,q}(\Omega)}
&\leq Ce^{-t}(1+t^{-m/2})\|N(\textstyle\frac{\tau}{2})\|_q+C \notag \\
&\qquad  +C\int_{0}^{t}e^{-(t-s)} \bigl(1+(t-s)^{-\frac{m+1}{2}}\bigr) \|\bigl(NA^{-1}\nabla A\bigr)
(\textstyle\frac{\tau}{2}+s)\|_q ds \label{varconstN} \\
&\leq C(1+\tau^{-m/2})+CM_q\int_{0}^{\infty}e^{-s} \bigl(1+s^{-(m+1)/2}\bigr)ds =: \tilde C(\tau). \notag
\end{align}
Using Sobolev's imbedding again, taking $q$ close to $2^-$ and $m$ close to $1^-$, we deduce
$$
\sup_{t\in (\tau,\infty)} \|N(t)\|_p<\infty,\quad 1\le p<\infty.
$$
Due to (\ref{NablaApbound}), this implies that (\ref{Mqbound}) is true for any $q\in [1,\infty)$.
The argument in (\ref{varconstN}) then yields 
$$
\sup_{t\in (\tau,\infty)} \|N(t)\|_{W^{m,q}(\Omega)}<\infty,\quad 0<m<1,\ 1\le q<\infty
$$
and (\ref{boundNinfty}) follows by a further application of Sobolev imbeddings.
\hfill $ \Box $

\section{The case of a square domain: Proof of Theorem~\ref{global2}}
Let us consider the square $\Omega=(0, L)^2$.
Theorem~\ref{global2} is a consequence of Theorem~\ref{global} and of the following two lemmas,
where we estimate the constants $K$ and $\mu$ in inequalities (\ref{EllEstim}) and (\ref{PoincSob})

\begin{lemma}\label{const2}
For $\Omega=(0, L)^2$, the constant $K$ in estimate (\ref{EllEstim}) satisfies $K\leq 12$.
\end{lemma}

\begin{proof}
Recall that the constant $K=K(\Omega)$ is determined from inequality (\ref{nabla4b}) in the proof of Lemma~\ref{grad}.
Repeatedly using the Cauchy-Schwarz inequality, we first compute
\begin{align*}
\bigl|2D^2 u+(\Delta u)I\bigr|_1
&=\max \bigl\{  |3u_{xx}+u_{yy}|+2|u_{xy}|, |u_{xx}+3u_{yy}|+2|u_{xy}|\bigr\}\\
&=2|u_{xy}|+\max \bigl\{  |3u_{xx}+u_{yy}|, |u_{xx}+3u_{yy}||\bigr\} \\
&\le2|u_{xy}|+ \bigl( 10(u_{xx}^2+u_{yy}^2)\bigr)^{1/2}
\end{align*}
and then
$$\bigl|2D^2 u+(\Delta u)I\bigr|_1^2 \le 12\bigl(u_{xx}^2+u_{yy}^2+2u_{xy}^2\bigr).$$
Consequently,
\begin{equation}
\label{Applic1}
\int_{\Omega}\bigl|2D^2 u+(\Delta u)I\bigr|_1^2 dx\leq 12 \bigl(\|u_{xx}\|_2^2+\|u_{yy}\|_2^2+2\|u_{xy}\|_2^2\bigr).
\end{equation}

We next claim that
\begin{equation}
\label{Applic2}
\|u_{xx}\|_2^2+\|u_{yy}\|_2^2+2\|u_{xy}\|_2^2 = \|\Delta u\|^2_2.
\end{equation}
Indeed, using the Fourier expansion $u=\sum_{j,k\geq 0}a_{jk}\cos\frac{j\pi x}{L}\cos\frac{k\pi y}{L}$, we have
\begin{align*}
 \|u_{xx}\|_2^2+\|u_{yy}\|_2^2&=
 \frac{\pi^4}{2L^2}\sum_{j\geq 1}j^4a_{j0}^2+
 \frac{\pi^4}{2L^2}\sum_{k\geq 1}k^4a_{0k}^2+
 \frac{\pi^4}{4L^2}\sum_{j,k\geq 1}(j^4+k^4)a_{jk}^2,\\
\|u_{xy}\|_2^2&=
 \frac{\pi^4}{4L^2}\sum_{j,k\geq 1}j^2k^2 a_{jk}^2,\\
\|u_{xx}+u_{yy}\|_2^2&=
 \frac{\pi^4}{2L^2}\sum_{j\geq 1}j^4a_{j0}^2+
 \frac{\pi^4}{2L^2}\sum_{k\geq 1}k^4a_{0k}^2+
 \frac{\pi^4}{4L^2}\sum_{j,k\geq 1}(j^2+k^2)^2a_{jk}^2,
\end{align*}
which implies (\ref{Applic2}).

The conclusion follows by combining  (\ref{nabla4b}), (\ref{Applic1}) and (\ref{Applic2}).
\end{proof}

\begin{lemma}\label{const1}
For $\Omega=(0, L)^2$, the constant $\mu$ in estimate (\ref{PoincSob}) satisfies 
$\mu\leq \sqrt{ \frac{3}{2}}$.
\end{lemma}

\begin{proof}
Let $u$ be a smooth function with zero average. There exist mean values $\zeta(x), \xi(y)$ in $[0, L]$ such that
$$\frac{1}{L}\int_{0}^{L}u(s,y)ds=u(\xi(y),y), \quad \frac{1}{L}\int_{0}^{L}u(x,t)dt=u(x,\zeta(x)).$$
It follows that
\begin{align*}
 u(x,y)&=\frac{1}{L}\int_{0}^{L}u(s,y)ds+\int_{\xi(y)}^{x}u_x(s,y)ds\\
&=\frac{1}{L}\int_{0}^{L}u(x,t)dt+\int_{\zeta(x)}^{y}u_y(x,t)dt.
\end{align*}
Consequently, 
\begin{align*}
\int_{\Omega}u^2(x,y)dxdy=&\frac{1}{L^2}\bigg(\int_{\Omega}u(x,y)dxdy\bigg)^2\\
 &+\int_{\Omega}\bigg(\frac{1}{L}\int_{0}^{L}u(s,y)ds\int_{\zeta(x)}^{y}u_y(x,t)dt \bigg)dxdy\\
&+\int_{\Omega}\bigg(\frac{1}{L}\int_{0}^{L}u(x,t)dt \int_{\xi(y)}^{x}u_x(s,y)ds \bigg)dxdy\\
&+\int_{\Omega}\bigg(\int_{\xi(y)}^{x}u_x(s,y)ds \int_{\zeta(x)}^{y}u_y(x,t)dt \bigg)dxdy.
\end{align*}
Since $u$ has zero average, it follows that
$$
\|u\|_2^2\leq \frac{1}{L}\|u\|_1(\|u_x\|_1+ \|u_y\|_1)+ \|u_x\|_1\|u_y\|_1.
$$
Moreover, by \cite[Theorem 3.2]{AD04}, we have $\|u\|_1\le (L/\sqrt{2})\|\nabla u\|_1$, hence
$$
\|u\|_2^2\leq \frac{1}{\sqrt 2}\|\nabla u\|_1(\|u_x\|_1+ \|u_y\|_1)+ \frac{1}{4}(\|u_x\|_1+\|u_y\|_1)^2.
$$
Using the elementary inequality $\|u_x\|_1+ \|u_y\|_1\le \sqrt 2\|\nabla u\|_1$, we deduce that
\begin{align*}
\|u\|_2\leq \sqrt{\frac{3}{2}}\|\nabla u\|_1.
\end{align*}
Since $C^1(\bar\Omega)$ is dense in $W^{1,1}(\Omega)$, Lemma~\ref{const1} follows by 
letting $u=N-|\Omega|^{-1}\int_{\Omega}N(x,y)dxdy$.
\end{proof}

\begin{remark}\label{remrem}
Let us briefly justify the generalization in Remark 1.2. 
Conditions $({\mathcal H}_1), ({\mathcal H}_2), ({\mathcal H}_5)$ are natural.
By condition $({\mathcal H}_3)$ and $N^{1-\delta}\leq \varepsilon N+ 
C(\varepsilon)$, one can show that $\|N(t)\|_1$ is uniformly bounded. And then, one has following estimates
$$|g(A,N) \log N|\leq C_1(A) N+C_2(A) \eqno(a) $$
$$|g(A,N)|N\leq \varepsilon N^2+ C_3(\varepsilon, A)\eqno(b)$$
Inequality $(a)$ and boundedness of $\|N(t)\|_1$ imply (\ref{c2}). Inequality $(b)$ implies (\ref{c5}).
By condition  $({\mathcal H}_4)$, one has 
$$|f(A,N)|^2\leq C_4(A)N^2+ C_5(A)$$
which implies  inequality (\ref{c1})
\end{remark}

\section{The case $\chi\le 1$: Proof of Theorem~\ref{global3}}

For $\chi\le 1$, a basic entropy estimate is given by the following.

\begin{lemma}\label{prioriNlog}
Assume $0<\chi\le 1$ and $A_{\rm max}=1$. Then
\begin{equation}
\label{boundNlogN}
\sup_{0<t<T^*} \int_\Omega N|\log N|\,dx<\infty.
\end{equation}
\end{lemma}

\begin{proof}
Following~\cite[Section~4]{Bi2}, we define
$$Y(t)=\int_\Omega N(\log N-c\log A)\,dx.$$
Letting $F= \psi N A(1-A)+ \tilde{A}-A$ and $G=\omega(1-N)$, we have
\begin{align}
Y'(t)&=\int_\Omega N_t(\log N-c\log A)\,dx+\int_\Omega (N_t-cNA^{-1}A_t)\,dx \notag \\
&=Q+\int_\Omega (\log N-c\log A)G\,dx+\int_\Omega G\,dx-c\int_\Omega NA^{-1}F\,dx, \label{Yprime}
\end{align}
where
\begin{align*}
Q&=\int_\Omega \Bigl[\Delta N-\nabla\cdot(N\nabla\vartheta(A))\Bigr](\log N-c\log A)\,dx
-c\eta\int_\Omega NA^{-1}\Delta A\,dx \\
&=-\int_\Omega N^{-1}|\nabla N|^2\, dx+(\chi+c(1+\eta))\int_\Omega A^{-1}\nabla N\cdot\nabla A
-c(\chi+\eta)\int_\Omega NA^{-2}|\nabla A|^2\, dx \\
&=-\int_\Omega N^{-1}A^{-2}\Bigl[A^2|\nabla N|^2+c(\chi+\eta)N^2|\nabla A|^2-(\chi+c(1+\eta))NA\nabla N\cdot\nabla A
\Bigr]\, dx.
\end{align*}
We see that $Q\le 0$ provided the discriminant $\delta:=(\chi+c(1+\eta))^2-4c(\chi+\eta)$ is nonpositive,
that is 
\begin{align}\label{cc}
(1+\eta)^2c^2-2c(\chi+2\eta-\chi\eta)+\chi^2\leq 0.
 \end{align}
There exists a constant $c$ satisfying (\ref{cc}) if and only if 
$$|\chi+2\eta-\chi\eta|\geq (1+\eta)\chi,$$
which is equivalent to  $\chi\le 1$.

On the other hand, we have $G\le \omega$ and, since $A\le 1$ due to (\ref{ineqAmax}), $F\ge -A$. 
Using (\ref{ineqAmin}), (\ref{ineqN1b}) and the inequality $\log s\le s$ ($s>0$),
we deduce from (\ref{Yprime}) that
$$Y'(t)\le -\omega Y+\omega \int_\Omega (\log N-c\log A)\,dx+\omega|\Omega|+c\int_\Omega N\,dx\le -\omega Y+C.$$
We infer that $\sup_{0<t<T^*}Y(t)<\infty$, so that (\ref{boundNlogN}) follows from (\ref{ineqAmax}), (\ref{ineqN1b}).
\end{proof}

With the bound (\ref{prioriNlog}) at hand, we can now use the following better Sobolev type inequality
in place of Lemma~\ref{leminequa}.

\begin{lemma}\label{leminequa2}
For any $\eps>0$, there holds
\begin{align}\label{PoincSobLog}
\|u\|_2^2\leq \eps\int_{\Omega}u|\log u|\,dx\ \int_{\Omega}\frac{|\nabla u|^2}{u}dx+C(\Omega,\|u\|_1,\eps)
\end{align}
for any uniformly positive function $u\in H^1(\Omega)$.

\end{lemma}
\begin{proof}
Fix $k>1$. Since $\nabla[u-k]_+=\chi_{\{u>k\}}\nabla u$, we deduce from (\ref{PoincSob2}) that
$$
\int_\Omega [u-k ]_+^2 dx 
\leq  \mu^2 \left( \int_{u>k}  | \nabla u  | dx \right)^2 
+  |\Omega|^{-1}\left( \int_\Omega [u-k ]_+ dx \right)^2.$$
On the other hand, we have $u^2\le [u-k]_+^2 + 2ku$
(just consider separately the cases $u \le$ or $> k$).
Therefore, by integration, we obtain
\begin{equation}
\label{eq:control-ulogu}
\begin{array}{rcl} 
{ \displaystyle{ \int_\Omega u^2 dx }}
& { \leq }& \displaystyle{  \mu^2 \left( \int_{u>k} | \nabla u | dx \right)^2 
+   |\Omega|^{-1} \|u\|_1^2  + 2k\|u\|_1 } \vspace{5pt} \\
& \leq & \displaystyle{  \mu^2 \int_{u>k} \frac{ | \nabla u |^2 }{ u } dx \cdot 
\int_{u>k} u  dx + |\Omega|^{-1} \|u\|_1^2 + 2k\|u\|_1 } \vspace{5pt} \\
& \leq & \displaystyle{  \mu^2 \int_\Omega \frac{ | \nabla u |^2 }{ u } dx \cdot 
\frac{ \displaystyle{ \int_\Omega u|\log u|dx}} {\log k} 
+ |\Omega|^{-1}  \|u\|_1^2 + 2k\|u\|_1} 
\end{array}
\end{equation}
and the Lemma follows by taking $k=\exp(\mu^2/\eps)$.
\end{proof}

{\it Proof of Theorem~\ref{global3}.}
We claim that
\begin{equation}\label{ineqNablaA2}
\sup_{t\in (0,T^*)}\|\nabla A(t)\|_2^2 <\infty
\end{equation}
and
\begin{equation}
\label{ineqDeltaA2}
\int_s^t\|\Delta A(\tau)\|_2^2\,d\tau\leq C(1+t-s),\quad 0<s<t<T^*.
\end{equation}

Fix $\eps>0$. By Lemmas~\ref{prioriNlog} and \ref{leminequa2}, we have
\begin{equation}
\|N\|_2^2\le \eps \int_{\Omega}\frac{|\nabla N|^2}{N}dx+C(\eps),\quad 0<t<T^*.
\label{ineqN2b}
\end{equation}
Using estimate (\ref{ineqN2b}) instead of (\ref{PoincSob0}) in the proof of Lemma~\ref{lembound}(i), 
we obtain the differential inequality~(\ref{ineqphi12})
where the constants $a_i$ are defined by formulae (\ref{defa12}) and (\ref{defa34})
with $ \mu^2N_{1,\max}$ replaced by~$\eps$.
Like before, we assume $a_2,a_3<0$ and choose $\sigma=-a_1/a_3>0$. Then we have $a_2+\sigma a_4<0$ provided
$a_1a_4<a_2a_3$, which is now equivalent to
$$K \chi^4A_{\min}^{-4}A^4_{\max}(A_{\max}-A_{\min})^2
\psi^2   \eps <
16\eps_1(\eta-\eps_1)\bigl(4(1-\eps_2)\eps_2-\eps_3\eps\bigr)\eps_3.
$$
Choosing $\eps_1=\eta/2$, $\eps_2=1/2$, $\eps_3=1/(2\eps)$, the condition becomes
$K \chi^4A_{\min}^{-4}A^4_{\max}(A_{\max}-A_{\min})^2\psi^2   \eps < \eta^2\eps^{-1}$
which is verified for $\eps$ suitably small. Estimates (\ref{ineqNablaA2})-(\ref{ineqDeltaA2}) then follow
as in the proof of Lemma~\ref{lembound}(ii).

The rest of the proof of the theorem then relies on Lemma~\ref{bounda} similarly as in the proof of Theorem~\ref{global}.
\hfill $ \Box $

\section{Conclusion}

In this article, we have considered a nonlinear, strongly coupled, parabolic system arising in the modeling of burglary in residential areas.
The system involves two spatio-temporal unknowns: the attractivity value of the property and the criminal density.
The system is of chemotaxis-type and involves a logarithmic sensivity function and specific interaction and relaxation terms.
This model has appeared in~\cite{Pit10}, as a modification of the model of Short et al.~\cite{SDPTBBC}. In~\cite{Pit10} some results about linearized stability/instability of the homogeneous steady-states 
and some numerical simulations suggesting the existence of hotspots were given.
However, the (local and) global existence of solutions was left open.

In this article, under suitable assumptions on the data of the problem, we have given a rigorous proof of the existence of a global 
and bounded, classical solution,
thereby solving the problem left open in \cite{Pit10}.
In the range of anti-diffusion parameter $\chi$ relevant for the criminological model, our sufficient condition for global existence roughly says that,
at the initial time, the product of the oscillation of the attractivity value and of the total criminal population should not be two large.
Our proofs are based on the construction of approximate entropies and on the use of various functional inequalities.

We have also provided explicit numerical conditions for global existence when the domain is a square,
including concrete cases involving values of the parameters
which are expected to be physically relevant.
In such cases, in the range of diffusion parameter $\eta$ used by~\cite{SDPTBBC}, our global existence conditions are compatible with magnitudes of attractivity which are up to about twice those of their static component. It is suggested in~\cite{Pit10} that a ratio of order 10 (instead of 2) might be desirable, even for smaller values of $\eta$ that the ones we consider here. However, this is still beyond the range in which we can rigorously prove global existence (and actually, global existence for such systems need not be taken for granted, as shown by the existing chemotaxis literature).

\end{document}